\documentclass[reqno]{amsart}
\usepackage{amsmath, amsthm, amssymb, mathabx}
\begin{document}
\newtheorem{theo}{Theorem}
\newtheorem{defin}[theo]{Definition}
\newtheorem{rem}[theo]{Remark}
\newtheorem{lem}[theo]{Lemma}
\newtheorem{cor}[theo]{Corollary}
\newtheorem{prop}[theo]{Proposition}
\newtheorem{exa}[theo]{Example}
\newtheorem{exas}[theo]{Examples}
%
%
\subjclass[2010]{Primary 35J60  Secondary 35J50}
\keywords{Indefinite functional; Variant fountain theorem; Superquadratic nonlinearity.}
\thanks{}
\title[On a superquadratic elliptic system]{On a superquadratic elliptic system with strongly indefinite structure}

\author[C. J. Batkam]{Cyril Joel Batkam}
\address{Cyril J. Batkam \newline
D\'epartement de math\'ematiques,
\newline
Universit\'e de Sherbrooke,
\newline
Sherbrooke, (Qu\'ebec),
\newline
J1K 2R1, CANADA.}
\email{cyril.joel.batkam@usherbrooke.ca}

\maketitle
\begin{abstract}
In this paper, we consider the elliptic system
\begin{equation*}
    \left\{
      \begin{array}{ll}
        -\Delta u=g(x,v)\,\, \textnormal{in }\Omega, & \hbox{} \\
        -\Delta v=f(x,u)\,\,\textnormal{in }\Omega, & \hbox{} \\
        u=v=0\textnormal{ on }\partial\Omega, & \hbox{}
      \end{array}
    \right.
\end{equation*}
where $\Omega$ is a bounded smooth domain in $\mathbb{R}^N$, and $f$ and $g$ satisfy a general superquadratic condition. By using variational methods, we prove the existence of infinitely many  solutions. Our argument relies on the application of a generalized variant fountain theorem for strongly indefinite functionals. Previous results in the topic are improved.
\end{abstract}

%
\section{Introduction}
In this article, we study the existence of multiple solutions of certain superquadratic elliptic systems of the form
\begin{equation}\label{s}\tag{S}
    \left\{
      \begin{array}{ll}
        -\Delta u=g(x,v)\,\, \textnormal{in }\Omega, & \hbox{} \\
        -\Delta v=f(x,u)\,\,\textnormal{in }\Omega, & \hbox{} \\
        u=v=0\textnormal{ on }\partial\Omega, & \hbox{}
      \end{array}
    \right.
\end{equation}
where $\Omega$ is a bounded smooth domain in $\mathbb{R}^N$, $N\geq3$, and the functions $f,g:\Omega\times\mathbb{R}\to\mathbb{R}$ are continuous and superlinear. Such systems describe steady state  solutions of reaction-diffusion and hydrodynamical problems. The difficulties in studying \eqref{s} originate mainly in two facts. First, the associated energy functional is strongly indefinite, in the sense that it is neither bounded from above nor from below, even on subspaces of finite dimension or codimension. Therefore, the usual critical point theorems cannot be applied. Second, due to the growth conditions on $f$ and $g$ below, the energy functional associated with \eqref{s} is not defined on the Sobolev space $H_0^1(\Omega)\times H_0^1(\Omega)$. We will use fractional Sobolev spaces in order to apply variational methods.
\par Elliptic systems leading to strongly indefinite functionals have been studied by many authors. See, for instance \cite{Ba-Cla,H-V,Figue-Fel,S-W1,BenRabi,Lupo,F-D} and the references therein. In a recent paper \cite{S-W1}, Szulkin and Weth  considered \eqref{s} with $f$ and $g$ both subcritical and odd, and they assumed among others that the mappings
$u\mapsto f(x,u)/|u|$ and $u\mapsto g(x,u)/|u|$ are strictly increasing in $(-\infty,0)\cup(0,+\infty)$. By developing a Nehari manifold method for strongly indefinite functionals, they obtained the existence of infinitely many solutions. We recall that if both $f$ and $g$ are subcritical, then the energy functional associated to \eqref{s} is well defined on the space $E=H_0^1(\Omega)\times H_0^1(\Omega)$, and has the form
\begin{equation*}
    J(u)=\frac{1}{2}\|u^+\|^2-\frac{1}{2}\|u^-\|^2-\int_\Omega\Big(F(x,u_1)+G(x,u_2)\Big)dx,
\end{equation*}
where
\begin{equation*}
    E=E^+\oplus E^-\ni u=(u_1,u_2)=u^++u^-,\quad u^\pm\in E^\pm=\{u\in E\,;\, u_2=\pm u_1\}.
\end{equation*}
Let the following set introduced by Pankov in \cite{Pan}.
\begin{equation*}
    \mathcal{M}=\big\{w\in E\backslash E^-\,|\, \big<J'(w),w\big>=0\text{ and }\big<J'(w),z=0\big>=0\,\,\forall z\in E^-\big\}.
\end{equation*}
The argument in \cite{S-W1} relies on the observation that for every $u\in E\backslash E^-$, the set $E^-\oplus \mathbb{R}^+u$ intersects $\mathcal{M}$ at exactly one point, namely $\widehat{m}(u)$. This allows the authors to reduce the problem on the manifold $\mathcal{M}$, and then on the unit sphere $S^+$ of $E^+$, where they can applied a classical multiplicity critical point theorem. If $u\mapsto f(x,u)/|u|$ or $u\mapsto g(x,u)/|u|$ is not strictly increasing in $(-\infty,0)\cup(0,+\infty)$, then $\widehat{m}(u)$ need not be unique, and their argument collapses.
\par The main goal of this paper is to extend the result of \cite{S-W1} by considering more general growth conditions on $f$ and $g$, and by only requiring the above mappings to be increasing. Our precise assumptions on $f$ and $g$ are the following.
\begin{enumerate}
  \item [$(H_1)$]$f,g\in\mathcal{C}(\Omega\times\mathbb{R})$ and there is a constant $C>0$ such that
\begin{equation*}
    |f(x,u)|\leq C(1+|u|^{p-1}) \textnormal{ and } |g(x,u)|\leq C(1+|u|^{q-1}),
\end{equation*}
for all $(x,u)$, where $p,q>2$ satisfy
\begin{equation*}
    \frac{1}{p}+\frac{1}{q}>1-\frac{2}{N}.
\end{equation*}
Furthermore, in case $N\geq5$ we impose
\begin{equation*}
    \frac{1}{p}>\frac{1}{2}-\frac{2}{N}\quad\textnormal{and}\quad \frac{1}{q}>\frac{1}{2}-\frac{2}{N}.
\end{equation*}
  \item [$(H_2)$] $\frac{1}{2}uf(x,u)\geq F(x,u)\geq0\text{ and } \frac{1}{2}ug(x,u)\geq G(x,u)\geq0,\,\, \forall (x,u).$ \\
  \item [$(H_3)$] $F(x,u)/u^2\to\infty$ and $G(x,u)/u^2\to\infty$ uniformly in $x$ as $|u|\to\infty$.\\
  \item [$(H_4)$]  $u\mapsto f(x,u)/|u|$ and $u\mapsto g(x,u)/|u|$ are increasing in $(-\infty,0)\cup(0,+\infty)$.\\
  \item [$(H_5)$] $f(x,-u)=-f(x,u)$ and $g(x,-u)=-g(x,u)$ for all $(x,u)$.
\end{enumerate}
Before we state our main result, we recall the following definition.
\begin{defin}
We say that $(u,v)$ is a strong solution of \eqref{s} if $u\in W^{2,p/(p-1)}(\Omega)\cap W_0^{1,p/(p-1)}(\Omega)$, $v\in W^{2,q/(q-1)}(\Omega)\cap W_0^{1,q/(q-1)}(\Omega)$ and $(u,v)$ satisfies
\begin{equation*}
    \left\{
      \begin{array}{ll}
        -\Delta u=g(x,v)\,\, \textnormal{ a.e. in }\Omega, & \hbox{} \\
        -\Delta v=f(x,u)\,\,\textnormal{ a.e. in }\Omega. & \hbox{}
        \end{array}
    \right.
\end{equation*}
\end{defin}
The main result of the paper is the following.
\begin{theo}\label{mainresult}
Under assumptions $(H_1)-(H_5)$, \eqref{s} has infinitely many pairs of strong solutions $\pm(u_k,v_k)$ such that $\|u_k\|_{L^\infty(\Omega)}\to\infty$ or $\|v_k\|_{L^\infty(\Omega)}\to\infty$, as $k\to\infty$.
\end{theo}
\par As far as we know, Theorem \ref{mainresult} is new under assumptions $(H_1)-(H_5)$. In \cite{F-W}, Felmer and Wang obtained the same result by using the Ambrosetti-Rabinowitz superquadratic condition: $\exists\gamma_1,\gamma_2>2$ and $\exists R>0$ such that
\begin{equation}\label{AR}\tag{AR}
     0<\gamma_1 F(x,u)\leq uf(x,u)\text{ and } 0<\gamma_2 G(x,u)\leq ug(x,u)\text{ for }|u|>R.
\end{equation}
They reduced the problem to a semi-finite situation, and used condition $(AR)$ to verify that the energy functional satisfies a strong version of the usual Palais-Smale condition, which was crucial for their argument. We recall that \eqref{AR} implies $f(x,u)\geq c|u|^{\gamma_1}$ and $g(x,u)\geq c|u|^{\gamma_2}$ for $|u|$ large, hence it is stronger than $(H_3)$. We point out that without \eqref{AR} we do not know if any Palais-Smale sequence of the energy functional is bounded. In this paper we do not use any reduction method. Our approach relies on a generalized variant fountain theorem for strongly indefinite functionals, established by the author and Colin in \cite{B-C}, which combines the $\tau$-topology introduced by Kryszewski and Szulkin \cite{K-S}, with the idea of the monotonicity trick developed by Jeanjean \cite{LJJ}. It also has the advantage that it produces bounded Palais-Smale sequences of the energy functional.
\par In \cite{H-V}, Hulshof and van der Vorst obtained the existence of at least one nontrivial solution of \eqref{s} under condition \eqref{AR}, which
 was mainly used to verify that the energy functional has a linking geometry in the sense of Benci and Rabinowitz \cite{BenRabi}, and also satisfies the Palais-Smale condition. A similar result was obtained independently by de Figueiredo and Felmer in \cite{Figue-Fel}.
\par The paper is organized as follows. Section \ref{section2} contained the variational framework for the study of \eqref{s}. The proof of Theorem \ref{mainresult} will be given in section \ref{section3}. In section \ref{section4}, we state a similar result concerning an indefinite semilinear elliptic equation.
\par Through the paper, $|\cdot|_r$ denotes the usual norm of the Lebesgue space $L^r(\Omega)$.

\section{Variational setting}\label{section2}

Consider the Laplacian as the operator
\begin{equation*}
    -\Delta:H^2(\Omega)\cap H_0^1(\Omega)\subset L^2(\Omega)\to L^2(\Omega),
\end{equation*}
and let $(\varphi_j)_{j\geq1}$ a corresponding system of orthogonal and $L^2(\Omega)$-normalized eigenfunctions, with eigenvalues $(\lambda_j)_{j\geq1}$. Then, writing
\begin{equation*}
    u=\sum\limits_{j=1}^\infty a_j\varphi_j,\quad \textnormal{with } a_j=\int_\Omega u\varphi_j dx,
\end{equation*}
we set, for $0\leq s\leq2$
\begin{equation*}
    E^s:=\big\{u\in L^2(\Omega)\, \big|\,\sum\limits_{j=1}^\infty\lambda_j^s|a_j|^2<\infty \big\}
\end{equation*}
and
\begin{equation*}
    A^s(u):=\sum\limits_{j=1}^\infty\lambda_j^{s/2}a_j\varphi_j,\quad \forall u\in D(A^s)=E^s.
\end{equation*}
One can verify easily that $A^s$ is an isomorphism onto $L^2(\Omega)$. We denote $A^{-s}:=(A^s)^{-1}$. It is well known (see Lions-Magenes \cite{Lions-Ma}) that the space $E^s$ is a fractional Sobolev space with the inner product
\begin{equation*}
    \big<u,v\big>_s=\int_\Omega A^suA^svdx.
\end{equation*}
We refer to the paper of Persson \cite{Per} for the proof of the following lemma.
\begin{lem}\label{compactness}
$E^s$ embeds continuously in $L^r(\Omega)$ for $s>0$ and $r\geq1$ satisfying $\frac{1}{r}\geq\frac{1}{2}-\frac{s}{N}.$ Moreover, the embedding is compact in the case of strict inequality.
\end{lem}
By assumption $(H_1)$, there exist $s,t>0$ such that $s+t=2$ and
\begin{equation}\label{pq}
    \frac{1}{p}>\frac{1}{2}-\frac{s}{N}\quad \textnormal{and}\quad \frac{1}{q}>\frac{1}{2}-\frac{t}{N}.
\end{equation}
We consider the functional
\begin{equation*}
    \Phi(u,v):=\int_\Omega A^suA^tvdx-\int_\Omega\Big(F(x,u)+G(x,v)\Big)dx,\quad (u,v)\in E^s\times E^t.
\end{equation*}
It follows from Lemma \ref{compactness} that the inclusions $E^s\hookrightarrow L^p(\Omega)$ and  $E^t\hookrightarrow L^q(\Omega)$ are continuous. This, together with the estimate
\begin{equation*}
   \Big|\int_\Omega A^suA^tvdx\Big|\leq  |A^su|_2|A^tu|_2=\|u\|_s\|v\|_t,
\end{equation*}
imply that the functional $\Phi$ above is well defined on $E:=E^s\times E^t$.\\
Now a standard argument shows that if assumption $(H_1)$ holds, then the functional $\Phi$ is of class $\mathcal{C}^1$ on $E$.\\
We say that $(u,v)\in E^s\times E^t$ is a weak solution of \eqref{s} if
\begin{multline*}
    \int_\Omega \big(A^suA^tk+A^shA^tv\big)dx-\int_\Omega\Big(hf(x,u)+kg(x,v)\Big)dx=0,\quad\forall(h,k)\in E^s\times E^t.
\end{multline*}
In order to recuperate from the critical points of the functional $\Phi$ (strong) solutions of \eqref{s}, we need the following regularity result due to de Figueiredo and Felmer \cite{Figue-Fel}.
\begin{lem}
If $(u,v)\in E^s\times E^t$ is a weak solution of \eqref{s}, then $(u,v)$ is a strong solution of \eqref{s}.
\end{lem}
\par We endow $E=E^s\times E^t$ with the inner product
\begin{equation*}
    \big<(u,v),(\phi,\varphi)\big>_{s\times t}=\big<u,\phi\big>_s+\big<v,\varphi\big>_t,\quad (u,v),(\phi,\varphi)\in E,
\end{equation*}
and the associated norm $\|(u,v)\|_{s\times t}^2=\big<(u,v),(u,v)\big>_{s\times t}$.
\par In the following we assume without loss of generality that $s\geq t$. One can easily verify that $E$ has the orthogonal decomposition \big(with respect to $\big<\cdot,\cdot\big>_{s\times t}$\big) $E=E^+\oplus E^-$, where
\begin{equation}\label{}
    E^+:=\big\{(u,A^{s-t}u)\,|\,u\in E^s\big\}\quad \textnormal{and}\quad E^-:=\big\{(u,-A^{s-t}u)\,|\,u\in E^s\big\}.
\end{equation}
If we denote by $P^\pm:E\to E^\pm$ the orthogonal projections, then a direct calculation yields
\begin{equation}\label{}
    P^\pm(u,v)=\frac{1}{2}\big(u\pm A^{t-s}v,v\pm A^{s-t}u\big),\quad \forall (u,v)\in E,
\end{equation}
and
\begin{equation}\label{phi}
    \Phi(u,v)=\frac{1}{2}\|P^+(u,v)\|_{s\times t}^2-\frac{1}{2}\|P^-(u,v)\|_{s\times t}^2-\int_\Omega\Big(F(x,u)+G(x,v)\Big)dx.
\end{equation}
Since both $E^-$ and $E^+$ are infinite-dimensional, the functional $\Phi$ is strongly indefinite, in the sense that it is neither bounded below nor above, even on subspaces of finite-dimension or finite-codimension. The study of $\Phi$ is therefore quite difficult, because the usual critical point theorems in \cite{R,W} cannot be applied directly.
\par Now we present the generalized variant fountain theorem we will apply in order to prove our main result.\\
\indent Let $Y$ be a closed subspace of a separable Hilbert space $X$ endowed with the inner product $\big<\cdot,\cdot\big>$ and the associated norm $\|\cdot\|$.
We denote by $P:X\rightarrow Y$ and $Q:X\rightarrow Z:=Y^\perp$ the orthogonal projections.\\
 We fix an orthonormal basis $(a_j)_{j\geq0}$ of $Y$, and we consider on $X=Y\oplus Z$ the $\tau$-topology introduced by Kryszewski and Szulkin in \cite{K-S}, that is, the topology associated with the norm
\begin{equation*}
    \vvvert u\vvvert:=\max\Big(\sum\limits_{j=0}^{\infty}\frac{1}{2^{j+1}}|\big<Pu,a_j\big>|,\|Qu\|\Big),\,\, u\in X.
\end{equation*}
Clearly we have $\|Qu\|\leq\vvvert u\vvvert\leq \|u\|$. Moreover, $\tau$ has the property that \big(see \cite{K-S} or \cite{W}\big): If $(u_n)\subset X$ is a bounded sequence, then
\begin{equation*}
     u_n
\stackrel{\tau}{\rightarrow}u  \Longleftrightarrow Pu_n \rightharpoonup Pu \,\ and \,\ Qu_n \rightarrow Qu.
\end{equation*}
Let $(e_j)_{j\geq0}$ be an orthonormal basis of $Z$. We adopt the following notations.
$$Y_k:=Y\oplus(\oplus_{j=0}^k{\mathbb R} e_j),\quad\quad \quad\quad  Z_k:=  \overline{\oplus_{j=k}^\infty\mathbb{R} e_j }, $$
$$ \displaystyle B_k:=\{u\in Y_k \, \bigl | \, ||u||\leq \rho_k \bigr.\}, \,\,\ \textnormal{with}\,\,\,\ \rho_k>0, \,\ k\geq2.   $$
\begin{theo}[Variant fountain theorem, Batkam-Colin \cite{B-C}]\label{vft}
Let the family of $\mathcal{C}^1$-functionals
\begin{equation*}
    \Phi_\lambda:X\rightarrow\mathbb{R},\quad\Phi_\lambda(u):=L(u)-\lambda J(u), \,\,\,\,\,\,\,\,\ \lambda\in[1,2],
\end{equation*}
such that
 \begin{enumerate}
\item [$(A_1)$] $\Phi_\lambda$ maps bounded sets to bounded sets uniformly for $\lambda\in[1,2]$, and $\Phi_\lambda(-u)=\Phi_\lambda(u)$ for every $(\lambda,u)\in[1,2]\times X$.
\item [$(A_2)$] $J(u)\geq0$ for every $u\in X$; $L(u)\rightarrow\infty$ or $J(u)\rightarrow\infty$ as $\|u\|\rightarrow\infty$.
\item [$(A_3)$] For every $\lambda\in[1,2]$, $\Phi_\lambda$ is $\tau$-upper semicontinuous  and $\Phi'_\lambda$ is weakly sequentially continuous.
\end{enumerate}
Let
$\Gamma_k(\lambda)$ be the class of maps $\gamma:B_k\rightarrow X$ such that
\begin{itemize}
  \item [(a)] $\gamma$ is odd and $\tau-$continuous, and $\gamma_{\mid_{\partial B_k}}=id,$
  \item [(b)] every $u\in int(B_k)$ has a $\tau-$neighborhood $N_u$ in $Y_k$ such that $(id-\gamma)(N_u\cap int(B_k))$ is contained in a finite-dimensional subspace of $X$,
   \item [(c)] $\Phi_\lambda(\gamma(u))\leq\Phi_\lambda(u)$ $\forall u\in B_k$.
\end{itemize}
If there are $0<r_k<\rho_k$ such that
\begin{equation*}
      b_k(\lambda) \, := \, \displaystyle \inf_{\substack{u \in Z_k \\ \|u \| = r_k}}\Phi_\lambda(u)\,\,\geq\,\,a_k(\lambda) \, := \, \displaystyle \sup_{\substack{u \in Y_k \\ \|u\| = \rho_k}} \Phi_\lambda(u), \quad\forall\lambda\in[1,2],
\end{equation*}
 then
\begin{equation*}
     c_{k}(\lambda) :=  \inf_{\gamma \in \Gamma_{k}(\lambda)} \sup_{u\in B_{k}} \Phi_\lambda
\bigl( \gamma(u)  \bigr)\,\geq\,b_k(\lambda),\quad \forall\lambda\in[1,2].
\end{equation*}
 Moreover, for a.e $\lambda\in[1,2]$ there exists a sequence $(u_k^n(\lambda))_n\subset X$ such that
\begin{equation*}
    \sup_{\substack{n}}\|u_k^n(\lambda)\|<\infty, \quad \Phi'_\lambda (u_k^n(\lambda))\rightarrow 0 \,\ and \,\ \Phi_\lambda(u_k^n(\lambda))\rightarrow c_k(\lambda)\,\,\textnormal{as}\,\,n\rightarrow\infty.
\end{equation*}
\end{theo}

\section{Proof of the main result}\label{section3}

Throughout this section we assume that $(H_1)-(H_5)$ hold.\\
 We define
\begin{equation*}
    X=E,\quad Y=E^-\quad  Z=E^+,\quad P=P^-\quad\textnormal{and } Q=P^+,
\end{equation*}
where $E=E^s\times E^t$, $E^\pm$ and $P^\pm$ are define in section \ref{section2} above. The functional $\Phi$ in \eqref{phi} then reads
\begin{equation}\label{phi1}
    \Phi(u,v)=\frac{1}{2}\|Q(u,v)\|_{s\times t}^2-\frac{1}{2}\|P(u,v)\|_{s\times t}^2-\int_\Omega\Big(F(x,u)+G(x,v)\Big)dx,
\end{equation}
$\forall (u,v)\in X.$
\par Let the family of functionals $\big\{\Phi_\lambda:X\to\mathbb{R}\,;\,\lambda\in[1,2]\big\}$ defined by
\begin{equation}\label{philambda}
    \Phi_\lambda(u,v)=\frac{1}{2}\|Q(u,v)\|_{s\times t}^2-\lambda\Big[\frac{1}{2}\|P(u,v)\|_{s\times t}^2+\int_\Omega\Big(F(x,u)+G(x,v)\Big)dx\Big].
\end{equation}
A standard argument shows that:
\begin{lem}
The conditions $(A_1)$ and $(A_2)$ of Theorem \ref{vft} are satisfied, with
\begin{equation*}
    L(u,v)=\frac{1}{2}\|Q(u,v)\|_{s\times t}^2,\quad J(u,v)=\frac{1}{2}\|P(u,v)\|_{s\times t}^2+\int_\Omega\Big(F(x,u)+G(x,v)\Big)dx.
\end{equation*}
Moreover, $\Phi'_\lambda$ is given by
\begin{multline}\label{phiprime}
    \big<\Phi'_\lambda(u,v),(h,k)\big>=\big<Q(u,v),(h,k)\big>_{s\times t}\\-\lambda\Big[\big<P(u,v),(h,k)\big>_{s\times t}+\int_\Omega\Big(hf(x,u)+kg(x,v)\Big)dx\Big],
\end{multline}
$\forall (u,v),(h,k)\in X.$
\end{lem}
We first show that condition $(A_3)$ of Theorem \ref{vft} is satisfied.
\begin{lem}
For every $\lambda\in[1,2]$, $\Phi_\lambda$ is $\tau-$upper semicontinous and $\Phi'_\lambda$ is weakly sequentially continuous.
\end{lem}
\begin{proof}
\begin{enumerate}
  \item Let $(u_n,v_n)\stackrel{\tau}{\rightarrow}(u,v)$ in $X$ and $\Phi_\lambda(u_n,v_n)\geq C\in\mathbb{R}$. It follows from the definition of $\tau$ that $(Q(u_n,v_n))_n$ is bounded. Since $F,G\geq0$, we deduce from the inequality $\Phi_\lambda(u_n,v_n)\geq C$ that $(P(u_n,v_n))_n$ is also bounded. Hence, up to a subsequence $(u_n,v_n)\rightharpoonup (u,v)$ in $X$ and  $(u_n,v_n)\to(u,v)$ a.e. in $\Omega$. It follows from Fatou's lemma and the weakly semicontinuity of the norm that $C\leq\Phi_\lambda(u,v)$. Hence, $\Phi_\lambda$ is $\tau-$upper semicontinous.
  \item Assume that $(u_n,v_n)\rightharpoonup (u,v)$ in $X=E^s\times E^t$. By \eqref{pq} and Lemma \ref{compactness}, the inclusion $X\hookrightarrow L^p(\Omega)\times L^q(\Omega)$ is compact. Therefore, $(u_n,v_n)\rightharpoonup (u,v)$ in $L^p(\Omega)\times L^q(\Omega)$. A standard argument based on the H\"{o}lder inequality and Theorem $A.2$ in \cite{W} shows that $\big<\Phi'_\lambda(u_n,v_n),(h,k)\big>\to\big<\Phi'_\lambda(u,v),(h,k)\big>$ for all $(h,k)\in X$. Hence, $\Phi'_\lambda$ is weakly sequentially continuous.
\end{enumerate}
\end{proof}
\par We recall that
\begin{equation*}
    Y=\big\{(u,-A^{s-t}u)\,\big|\,u\in E^s\big\}\quad\textnormal{and}\quad Z=\big\{(u,A^{s-t}u)\,\big|\,u\in E^s\big\}.
\end{equation*}
Let $(a_j)_{j\geq0}$ be an orthonormal basis of $E^s.$ Then $(A^{s-t}a_j)_{j\geq0}$ is an orthonormal basis of $E^t$. We define an orthonormal basis $(e_j)_{j\geq0}$ of $Z$ by setting
\begin{equation*}
    e_j:=\frac{1}{\sqrt{2}}\big(a_j,A^{s-t}a_j\big).
\end{equation*}
Let
\begin{equation*}
    Y_k=Y\oplus\big(\oplus_{j=0}^k\mathbb{R}e_j\big)\quad\textnormal{and}\quad Z_k=\overline{\oplus_{j=k}^\infty\mathbb{R}e_j}.
\end{equation*}
\begin{lem}\label{z}
There exist $(\lambda_n)_{n\geq0}\subset[1,2]$ and $(u_k^n,v_k^n)_{n\geq0}\subset X\backslash\{0\}$ such that
\begin{equation*}
   \lambda_n\to1,\quad \Phi'_{\lambda_n}(u_k^n,v_k^n)=0\quad\textnormal{and}\quad\Phi_{\lambda_n}(u_k^n,v_k^n)=c_k(\lambda_n)
\end{equation*}
for $k$ big enough.
\end{lem}
\begin{proof}
$(H_3)$ implies that for every $\delta>0$ there is $C_\delta>0$ such that
\begin{equation}\label{superquadratic}
    F(x,u)\geq \delta|u|^2-C_\delta\quad\text{and}\quad G(x,u)\geq \delta|u|^2-C_\delta,\quad \forall (x,u).
\end{equation}
\par Let $z\in Y_k$. Then $z=\big(u,A^{s-t}u\big)+\big(v,-A^{s-t}v\big)$, where $v\in E^s$ and $u\in E^s_k:=\oplus_{j=0}^k\mathbb{R}a_j$.
By \eqref{philambda} we have
\begin{align*}
    \Phi_\lambda(z)&=\|u\|^2_s-\lambda\|v\|^2_s-\lambda\int_\Omega\Big(F(x,u+v)+G(x,A^{s-t}(u-v))\Big)dx\\
&\leq \|u\|^2_s-\|v\|^2_s-\int_\Omega\Big(F(x,u+v)+G(x,A^{s-t}(u-v))\Big)dx\quad(\text{since }\lambda\geq1)\\
&\leq \|u\|^2_s-\|v\|^2_s-\delta\big(|u+v|^2_2+|A^{s-t}(u-v)|^2_2\big)+2C_\delta|\Omega|\quad(\text{in view of \eqref{superquadratic}})\\
&\leq \|u\|^2_s-\|v\|^2_s-C\delta\big(|u+v|^2_2+|u-v|^2_2\big)+2C_\delta|\Omega|\quad(\text{since }E^{s-t}\hookrightarrow L^2(\Omega))\\
&=\|u\|^2_s-\|v\|^2_s-2C\delta\big(|u|^2_2+|v|^2_2\big)+2C_\delta|\Omega| \quad(\text{by the parallelogram identity}).
\end{align*}
Since all the norms are equivalent on the finite-dimensional subspace $E^s_k$, there is a constant $c_1>0$ such that $c_1\|u\|_s\leq|u|_2$. Hence
\begin{equation*}
    \Phi_\lambda(z)\leq (1-c_2\delta)\|u\|^2_s-\|v\|^2_s-2C_\delta|\Omega|.
\end{equation*}
Choose $\delta>\frac{1}{c_2}$. Hence $\Phi_\lambda(z)\to-\infty$ uniformly in $\lambda\in[1,2]$ as $\|z\|_{s\times t}\to\infty$, and consequently $a_k(\lambda)<0$ for $\rho_k$ big enough.
\par Let $z\in Z_k$. Then $z=(u,A^{s-t}u)$ with $u\in \overline{\oplus_{j=k}^\infty\mathbb{R}a_j}$, and
\begin{align*}
    \Phi_\lambda(z)&=\frac{1}{2}\|z\|_{s\times t}^2-\lambda\int_\Omega\Big(F(x,u)+G(x,A^{s-t}u)\Big)dx\\
&\geq\frac{1}{2}\|z\|_{s\times t}^2-2\int_\Omega\Big(F(x,u)+G(x,A^{s-t}u)\Big)dx\quad(\text{since }\lambda\leq2).
\end{align*}
By $(H_1)$ there is a constant $C_1>0$ such that
\begin{equation*}
    |F(x,u)|\leq C_1(1+|u|^p)\quad\text{and}\quad |G(x,A^{s-t}u)|\leq C_1(1+|A^{s-t}u|^q).
\end{equation*}
Hence
\begin{equation*}
    \Phi_\lambda(z)\geq\|u\|^2_s-2C_1|u|_p^p-2C_1|A^{s-t}u|^q_q-4C_1|\Omega|.
\end{equation*}
We define
\begin{equation*}
    \beta_{1,k}:=\sup_{\substack{u\in\overline{\oplus_{j=k}^\infty\mathbb{R}a_j}\\\|u\|_s=1}}|u|_p,\quad \beta_{2,k}:=\sup_{\substack{v\in\overline{\oplus_{j=k}^\infty\mathbb{R}(A^{s-t}a_j)}\\\|v\|_t=1}}|v|_q,
\end{equation*}
and $\beta_k=\max\big\{\beta_{1,k};\beta_{2,k}\big\}$.\\
Then
\begin{equation*}
    \Phi_\lambda(z)\geq\|u\|^2_s-2C_1\beta_k^p\|u\|_s^p-2C_1\beta_k^q\|u\|^q_s-4C_1|\Omega|.
\end{equation*}
We assume without loss of generality that $q\leq p$ and we set
\begin{equation*}
    r_k:=\big(C_1p\beta_k^p\big)^{\frac{1}{2-p}}.
\end{equation*}
Then for $\|z\|_{s\times t}=\sqrt{2}\|u\|_s=r_k$ we have
\begin{equation}\label{betatilde}
    \Phi_\lambda(z)\geq \widetilde{b_k}:=K\beta_k^{\frac{2p}{2-p}}\Big[\Big(\frac{1}{4}-\frac{1}{p(\sqrt{2})^p}\Big)-A\beta_k^{\frac{2(q-p)}{2-p}}\Big]-4C_1|\Omega|,
\end{equation}
where $K,A>0$ are constant.\\
By using the argument in the proof of Lemma $3.8$ in \cite{W}, we easily show that $\beta_{1,k}\to0$ and $\beta_{2,k}\to0$ as $k\to\infty$. This implies that $\widetilde{b_k}\to\infty$ and $\Phi_\lambda(z)\to\infty$ uniformly in $\lambda\in[1,2]$, as $k\to\infty$, whenever $\|z\|_{s\times t}=r_k$. Hence $b_k(\lambda)\geq\widetilde{b_k}>0$ for $k$ big enough.
\par By Theorem \ref{vft}, $c_k(\lambda)\geq b_k(\lambda)$ and for a.e. $\lambda\in[1,2]$ there exists $z^n_k(\lambda)=(u^n_k(\lambda),v^n_k(\lambda))\in X$ such that
\begin{equation*}
    \sup_{\substack{n}}\|z^n_k(\lambda)\|_{s\times t}<\infty,\quad \Phi_\lambda(z^n_k(\lambda))\to c_k(\lambda)\quad\text{and}\quad \Phi'_\lambda(z^n_k(\lambda))\to0,\quad\text{as} \quad n\to\infty,
\end{equation*}
for $k$ big enough.\\
Now a standard argument shows that $(z^n_k(\lambda))_n$ has a convergent subsequence. Therefore, there exists $(u_k(\lambda),v_k(\lambda))\in X$ such that
$\Phi'_\lambda(u_k(\lambda),v_k(\lambda))=0$ and $\Phi_\lambda(u_k(\lambda),v_k(\lambda))=c_k(\lambda)$. It is then easy to conclude.
\end{proof}
\par Next we will show that the sequence $(u_k^n,v_k^n)_n$ above is bounded. The following technical lemma will be very helpful.
\begin{lem}\label{comparaison}
Let $\lambda\in[1,2]$. if $z_\lambda\neq0$ and $\Phi'_\lambda(z_\lambda)=0$, then $\Phi_\lambda(z_\lambda+w)\leq\Phi_\lambda(z_\lambda)$ for every
$w\in\mathbf{Z}_\lambda:=\big\{rz_\lambda+\theta\,|\, r\geq-1,\,\,\theta\in Y\big\}$.
\end{lem}
\begin{proof}
Let $z_\lambda=(u_\lambda,v_\lambda)$ and $w=rz_\lambda+\theta$, where $r\geq-1$ and $\theta=(\theta_1,\theta_2)\in Y$.\\
A direct calculation gives
\begin{align*}
    \Phi_\lambda(z_\lambda+w)&-\Phi(z_\lambda)=-\frac{\lambda}{2}\|\theta\|^2_{s\times t}+r\big(\frac{r}{2}+1\big)\|Qz_\lambda\|^2_{s\times t}\\
&-\lambda \Big[r\big(\frac{r}{2}+1\big)\|Pz_\lambda\|^2_{s\times t}+(1+r)\big<Pz_\lambda,\theta\big>_{s\times t}\Big]\\
&-\lambda\int_\Omega\Big(F\big(x,(1+r)u_\lambda+\theta_1\big)-F(x,u_\lambda)\Big)dx\\
&-\lambda\int_\Omega\Big(G\big(x,(1+r)v_\lambda+\theta_2\big)-G(x,v_\lambda)\Big)dx.
\end{align*}
Now $\big<\phi'_\lambda(z_\lambda),r\big(\frac{r}{2}+1\big)z_\lambda+(1+r)\theta\big>=0$ implies
\begin{align*}
    r\big(\frac{r}{2}+1\big)\|Qz_\lambda\|^2_{s\times t}-\lambda \Big[r\big(\frac{r}{2}+1\big)\|Pz_\lambda\|^2_{s\times t}+(1+r)\big<Pz_\lambda,\theta\big>_{s\times t}\Big]=\\
\lambda\int_\Omega\Big[\big((1+r)u_\lambda+\theta_1\big)f(x,u_\lambda)+\big((1+r)v_\lambda+\theta_2\big)g(x,v_\lambda)\Big]dx.
\end{align*}
Hence
\begin{align*}
   \Phi_\lambda(z_\lambda&+w)-\Phi(\lambda)=-\frac{\lambda}{2}\|\theta\|^2_{s\times t}\\
&+\lambda\int_\Omega\Big[\big((1+r)u_\lambda+\theta_1\big)f(x,u_\lambda)+F(x,u_\lambda)-F\big(x,(1+r)u_\lambda+\theta_1\big)\Big]dx\\
&+\lambda\int_\Omega\Big[\big((1+r)v_\lambda+\theta_2\big)g(x,v_\lambda)+G(x,v_\lambda)-G\big(x,(1+r)v_\lambda+\theta_2\big)\Big]dx.
\end{align*}
Following Liu \cite{Liu}, we define for an arbitrary $\varepsilon>0$,
\begin{equation*}
    f_\varepsilon(x,u)=f(x,u)+\varepsilon u^3\quad\text{and}\quad g_\varepsilon(x,u)=g(x,u)+\varepsilon u^3,\quad \forall (x,u)\in \Omega\times\mathbb{R}.
\end{equation*}
Using $(H_4)$, one easily verifies that the mappings $u\mapsto f_\varepsilon(x,u)/|u|$ and $u\mapsto g_\varepsilon(x,u)/|u|$ are strictly increasing in $\mathbb{R}\backslash\{0\}$. It then follows from Lemma $2.2$ in \cite{S-W} that
\begin{equation*}
    \big((1+r)u_\lambda+\theta_1\big)f_\varepsilon(x,u_\lambda)+F_\varepsilon(x,u_\lambda)-F_\varepsilon\big(x,(1+r)u_\lambda+\theta_1\big)<0,
\end{equation*}
\begin{equation*}
    \big((1+r)v_\lambda+\theta_2\big)g_\varepsilon(x,v_\lambda)+G_\varepsilon(x,v_\lambda)-G_\varepsilon\big(x,(1+r)v_\lambda+\theta_2\big)<0,
\end{equation*}
where $F_\varepsilon$ and $G_\varepsilon$ are the primitives of $f_\varepsilon$ and $g_\varepsilon$ respectively. \\
Letting $\varepsilon\to0$, we get
\begin{equation*}
    \big((1+r)u_\lambda+\theta_1\big)f(x,u_\lambda)+F(x,u_\lambda)-F\big(x,(1+r)u_\lambda+\theta_1\big)\leq0,
\end{equation*}
\begin{equation*}
    \big((1+r)v_\lambda+\theta_2\big)g(x,v_\lambda)+G(x,v_\lambda)-G\big(x,(1+r)v_\lambda+\theta_2\big)\leq0.
\end{equation*}
This completes the proof of the lemma.
\end{proof}
\begin{lem}
The sequence $\big(z_k^n=(u_k^n,v_k^n)\big)_n$ obtained in Lemma \ref{z} is bounded.
\end{lem}
\begin{proof}
We argue by contradiction.\\
 Assume that $(z_k^n)_n$ is unbounded. Then, up to a subsequence $\|z_k^n\|_{s\times t}\to\infty$, as $n\to\infty$. Let $w_k^n=(s_k^n,t_k^n)=z_k^n/\|z_k^n\|_{s\times t}$.
Up to a subsequence we may suppose that $w_k^n\rightharpoonup w_k=(s_k,t_k)$ in $X$ and $w_k^n\to w_k=(s_k,t_k)$ a.e..\\
If $w_k\neq0$, that is, if $s_k\neq0$ or $t_k\neq0$, then $|s_k^n|\|z_k^n\|_{s\times t}\to\infty$ or $|t_k^n|\|z_k^n\|_{s\times t}\to\infty$ as $n\to\infty$.
Now for $k$ large enough we have
\begin{align*}
    0<\frac{c_k(\lambda_n)}{\|z_k^n\|_{s\times t}^2}=\frac{\Phi_{\lambda_n}(z_k^n)}{\|z_k^n\|_{s\times t}^2}&\leq\frac{1}{2}\|Qw_k^n\|_{s\times t}^2-\frac{1}{2}\|Pw_k^n\|_{s\times t}^2\\
&-\int_\Omega\frac{F(x,\|z_k^n\|_{s\times t})}{|s_k^n\|z_k^n\|_{s\times t}|^2}|s_k^n|^2dx-\int_\Omega\frac{G(x,\|z_k^n\|_{s\times t})}{|t_k^n\|z_k^n\|_{s\times t}|^2}|t_k^n|^2dx.
\end{align*}
We then obtain, by using $(H_3)$ and Fatou's lemma the contradiction $0\leq-\infty$.\\
Hence $w_k=0$. Since $\Phi_{\lambda_n}(z_k^n)>0$ and $F,G\geq0$, we have that $\|Qw_k^n\|_{s\times t}\geq\|Pw_k^n\|_{s\times t}$. It is then clear by definition of $w_k^n$ that we cannot have $\|Qw_k^n\|_{s\times t}\to0$, as $n\to\infty$. Hence, there is a constant $\alpha>0$ such that $\|Qw_k^n\|_{s\times t}\geq\alpha$ up to a subsequence. By Lemma \ref{comparaison}, we have for every $r>0$
\begin{align*}
    c_k(2)\geq c_k(\lambda_n)=\Phi_{\lambda_n}(z^n_k)&\geq\Phi_{\lambda_n}(rQw^n_k)\geq\frac{1}{2}\alpha^2r^2\\
&-\lambda_n\int_\Omega\Big(F(x,rQ_1w_k^n)+G(x,rQ_2w_k^n)\Big)dx,\quad(\star)
\end{align*}
where $Qw_k^n=\big(Q_1w_k^n,Q_2w_k^n\big)$.\\
 Now, by Lemma \ref{compactness}, $Qw_k^n\to0$ in $L^p(\Omega)\times L^q(\Omega)$, as $n\to\infty$. It follows from $(H_1)$ and Theorem $A.2$ in \cite{W} that $F(x,rQ_1w_k^n)\to0$ and $G(x,rQ_2w_k^n)\to0$ in $L^1(\Omega)$. By taking the limit $n\to\infty$ in $(\star)$ we obtain
\begin{equation*}
    c_k(2)\geq\frac{1}{2}\alpha^2r^2,\quad\forall r>0.
\end{equation*}
This gives a contradiction if we fix $k$ and let $r\to\infty$.\\
Consequently, the sequence $(z_n^k)_n$ is bounded.
\end{proof}
\par We can now prove Theorem \ref{mainresult}.
\begin{proof}[Proof of Theorem \ref{mainresult}]
We consider the sequence $(z_n^k=(u_n^k,v_n^k))_n$ above. The relation
\begin{align*}
     \big<\Phi'(u_k^n,v_k^n)-\Phi'_{\lambda_n}(u_k^n,v_k^n),(h,k)\big>&=(\lambda_n-1)\big[\big<P(u_k^n,v_k^n),(h,k)\big>_{s\times t}\\
&+\int_\Omega\big(hf(x,u_k^n)+kg(x,v_k^n)\big)dx\big]
\end{align*}
implies that
\begin{equation*}
    \lim\limits_{n\to\infty}\Phi'(u_k^n,v_k^n)=\lim\limits_{n\to\infty}\Phi'_{\lambda_n}(u_k^n,v_k^n)=0.
\end{equation*}
Now, since $(c_k(\lambda_n))_n$ is nondecreasing and bounded from above, there exists $\alpha_k\geq c_k(2)$ such that $c_k(\lambda_n)\to\alpha_k$ as $n\to\infty$. It follows from the equality
\begin{align*}
    \Phi(u_k^n,v_k^n)-\Phi_{\lambda_n}(u_k^n,v_k^n)=(\lambda_n-1)\Big[\frac{1}{2}\|P(u_k^n,v_k^n)\|^2_{s\times t}+\int_\Omega\big(F(x,u_k^n)+G(x,v_k^n)\big)dx\Big]
\end{align*}
that
\begin{equation*}
    \lim\limits_{n\to\infty}\Phi(u_k^n,v_k^n)=\lim\limits_{n\to\infty}\Phi_{\lambda_n}(u_k^n,v_k^n)=\lim\limits_{n\to\infty}c_k(\lambda_n)=\alpha_k.
\end{equation*}
By repeating the argument in the proof of Lemma \ref{z}, we see that there exists $(u_k,v_k)\in X$ such that $\Phi'(u_k,v_k)=0$ and $\Phi(u_k,v_k)=\alpha_k$. But since $\alpha_k\geq c_k(2)\geq b_k(2)\geq\widetilde{b_k}\to\infty$, as $k\to\infty$ \big(where $\widetilde{b_k}$ is defined in \eqref{betatilde}\big), the proof is completed.
\end{proof}
\section{A semilinear elliptic problem}\label{section4}
In this section, we consider the semilinear elliptic problem
\begin{equation}\label{see}
\left\{
  \begin{array}{ll}
    -\Delta u-\mu u=f(x,u),\quad x\in\Omega, & \hbox{} \\
    u=0,\quad\text{on } \partial\Omega& \hbox{}
  \end{array}
\right.
\end{equation}
where $\mu$ is a real parameter. Let $0<\mu_1<\mu_2<\mu_3<\cdots$ be the eigenvalues of the problem
\begin{equation*}
\left\{
  \begin{array}{ll}
    -\Delta u=\mu u,\quad \text{in }\Omega, & \hbox{} \\
    u=0,\quad\text{on } \partial\Omega. & \hbox{}
  \end{array}
\right.
\end{equation*}
We have the following result.
\begin{theo}\label{result2}
Assume that the following conditions are satisfied.
\begin{enumerate}
  \item [$(f_1)$]$f\in\mathcal{C}(\Omega\times\mathbb{R})$ and there is a constant $C>0$ such that $|f(x,u)|\leq C(1+|u|^{p-1}) $, for some
  $2<p<2^\star$, where $2^\star=+\infty$ if $N=1,2$ and $2^\star=2N/(N-2)$ if $N\geq3$.
  \item [$(f_2)$] $\frac{1}{2}uf(x,u)\geq F(x,u)\geq0,\,\, \forall (x,u).$
  \item [$(f_3)$] $F(x,u)/u^2\to\infty$ uniformly in $x$ as $|u|\to\infty$.
  \item [$(f_4)$]  $u\mapsto f(x,u)/|u|$ is increasing in $(-\infty,0)\cup(0,+\infty)$.
  \item [$(f_5)$] $f(x,-u)=-f(x,u)$ for all $(x,u)$.
\end{enumerate}
 If $\mu_k<\mu<\mu_{k+1}$ for some $k\geq1$, then \eqref{see} has infinitely many pairs of solutions $\pm u_k$ such that $|u_k|_\infty\to\infty$ as $k\to\infty$.
\end{theo}
\begin{proof}
 By $(f_1)$, the energy functional associated with \eqref{see} is defined on the Sobolev space $H_0^1(\Omega)$ by
\begin{equation*}
    \Psi(u)=\frac{1}{2}\int_\Omega\big(|\nabla u|^2-\mu u^2\big)dx-\int_\Omega F(x,u)dx.
\end{equation*}
By the Poincar\'{e} inequality, $H_0^1(\Omega)$ is equipped with the inner product
\begin{equation*}
    \big<u,v\big>_0=\int_\Omega \nabla u\nabla vdx,
\end{equation*}
and the associated norm $\|u\|^2_0=\big<u,u\big>_0$.\\
Let $e_1,e_2,e_3,\cdots$ be the orthogonal eigenfunctions in $H_0^1(\Omega)$ corresponding to $\mu_1,\mu_2,\mu_3,\cdots$. If $\mu_k<\mu<\mu_{k+1}$ for some $k\geq1$, then $X=H_0^1(\Omega)$ has the orthogonal decomposition $X=Y\oplus Z$, where
\begin{equation*}
    Y=span\{e_1,e_2,\cdots,e_k\}\quad \text{and} \quad Z=Y^\perp.
\end{equation*}
Clearly, we can define an equivalent inner product in $X$ with associated norm $\|\cdot\|$ such that
\begin{equation*}
    \int_\Omega\big(|\nabla u|^2-\mu u^2\big)dx=\frac{1}{2}\|Qu\|^2-\frac{1}{2}\|Pu\|^2,
\end{equation*}
where $P:X\to Y$ and $Q:X\to Z$ are the orthogonal projections.\\
The functional $\Psi$ above then reads
\begin{equation*}
    \Psi(u)=\frac{1}{2}\|Qu\|^2-\frac{1}{2}\|Pu\|^2-\int_\Omega F(x,u)dx.
\end{equation*}
Let
\begin{equation*}
    \Psi_\lambda(u)=\frac{1}{2}\|Qu\|^2-\lambda\Big[\frac{1}{2}\|Pu\|^2+\int_\Omega F(x,u)dx\Big],\quad \lambda\in[1,2].
\end{equation*}
Evidently, we can equipped $(X,\|\cdot\|)$ with the $\tau-$topology, and it is easy to check that $\Psi_\lambda$ is $\tau-$upper semicontinuous and that $\Psi'_\lambda$ is weakly sequentially continuous.\\
The rest of the proof uses an argument similar to that in the proof of Theorem \ref{mainresult}, which is now simplified since $dim Y<\infty$.
\end{proof}
\begin{rem}
Theorem \ref{result2} extends Theorem 3.2 in \cite{S-W}, where the mapping $u\mapsto f(x,u)/|u|$ was supposed to be strictly increasing in $\mathbb{R}\backslash\{0\}$.
\end{rem}

%
%

%
%
\end{document}